\documentclass[10pt]{amsart} 
\usepackage[foot]{amsaddr}
\usepackage{amsmath,amssymb,bm, color,float}  
\usepackage{graphicx}

\usepackage{mathrsfs}
\usepackage{bbm}
\usepackage{bm} 
\usepackage{amsfonts,amssymb} 
\usepackage{multirow}
\usepackage{lineno}
\usepackage{color}
\usepackage{hyperref}
\usepackage{amsmath}
\usepackage{graphicx}
\usepackage{array}
\usepackage{lineno}
\usepackage{tikz}
\usepackage{amsfonts,amssymb}
\usepackage{dsfont}
\usepackage{pifont}

\def\O{\Omega}
\def\pa{\partial}

\newcommand{\bn}{{\bf n}}

\newcommand{\bv}{{\bf v}}

\def\pa{\partial}

\def\O{\Omega}

\def\bvarphi{{\boldsymbol{\varphi}}}

\def\3bar{{|\hspace{-.02in}|\hspace{-.02in}|}}

\def\3bar{{|\hspace{-.02in}|\hspace{-.02in}|}}

\numberwithin{equation}{section}
 \newtheorem{theorem}{Theorem}[section]
 \newtheorem{lemma}[theorem]{Lemma}

  \def\b#1{\mathbf{#1}} \def\tx#1{\hbox{#1}}
\def\a#1{\begin{align*}#1\end{align*}} \def\an#1{\begin{align}#1\end{align}}

\newtheorem{WG}{Weak Galerkin Algorithm}

\title[Weak Galerkin Finite Element]
{Weak Galerkin finite element methods for elliptic interface problems on nonconvex polygonal partitions}

  \author {Chunmei Wang$\dagger$} 
  \address{Department of Mathematics, University of Florida, Gainesville, FL 32611, USA. }
  \email{chunmei.wang@ufl.edu}
  \thanks{The research of Chunmei Wang was partially supported by National Science Foundation Grant DMS-2136380.}   \thanks{$\dagger$ The corresponding author. }
 
\author {Shangyou Zhang}
\address{Department of Mathematical Sciences,  University of Delaware, Newark, DE 19716, USA}   \email{szhang@udel.edu}

\begin{document}

\begin{abstract}This paper proposes a weak Galerkin (WG) finite element method for elliptic interface problems defined on nonconvex polygonal partitions. The method features a built-in stabilizer and retains a simple, symmetric, and positive definite formulation. An optimal-order error estimate is rigorously derived in the discrete 
$H^1$
  norm. Furthermore, a series of numerical experiments are provided to verify the theoretical results and to demonstrate the robustness and effectiveness of the proposed WG method for elliptic interface problems.

\end{abstract}
\keywords{
weak Galerkin, finite element methods, elliptic interface problems, weak gradient, polygonal partitions, nonconvex. }
 
\subjclass[2010]{65N30, 65N15, 65N12, 65N20}
 
\maketitle

\section{Introduction}

This paper is concerned with the development and analysis of a weak Galerkin (WG) finite element method for elliptic interface problems posed on both convex and nonconvex polygonal partitions. We consider the following model problem: find a function $u$ such that
\begin{eqnarray} 
-\nabla\cdot({\color{black}{a}}\nabla u)&=&f,\quad \text{in } \Omega,\label{model-1}\\
u&=&g,\quad \text{on } \partial\Omega\setminus\Gamma,\label{model-2}\\
{[[u]]}_{\Gamma}=u|_{\Omega_1}-u|_{\Omega_2}&=&g_D,\quad \text{on } \Gamma,\label{model-3}\\
{[[{\color{black}{a}}\nabla u\cdot \bn]]}_{\Gamma}
={\color{black}{a_1}}\nabla u|_{\Omega_1}\cdot \bn_1
+{\color{black}{a_2}}\nabla u|_{\Omega_2}\cdot \bn_2
&=&g_N,\quad \text{on } \Gamma,\label{model-4}
\end{eqnarray}
where $\Omega\subset\mathbb{R}^2$ is a bounded domain decomposed as
$\Omega=\Omega_1\cup\Omega_2$, with interface $\Gamma=\Omega_1\cap\Omega_2$.
Here, $a_i=a|_{\Omega_i}$ for $i=1,2$, and $\bn_1$ and $\bn_2$ denote the unit outward normal vectors to $\Gamma$ with respect to $\Omega_1$ and $\Omega_2$, respectively. The diffusion tensor $a$ is assumed to be symmetric and uniformly positive definite in $\Omega$. For simplicity of presentation, $a$ is taken to be a constant matrix; however, all theoretical results extend in a straightforward manner to sufficiently smooth variable coefficients.

\medskip

A weak formulation of \eqref{model-1}--\eqref{model-4} reads as follows: find $u\in H^1(\Omega)$ such that $u=g$ on $\partial\Omega\setminus\Gamma$, ${[[u]]}_{\Gamma}=g_D$ on $\Gamma$, and
\begin{equation}\label{weak-formula}
({\color{black}{a}}\nabla u,\nabla v)=(f,v)+\langle g_N,v\rangle_\Gamma,
\quad \forall v\in H_0^1(\Omega),
\end{equation}
where $H_0^1(\Omega)=\{v\in H^1(\Omega): v=0 \text{ on } \partial\Omega\}$.

Elliptic interface problems arise in numerous scientific and engineering applications, including biological systems \cite{KLL2009}, material science \cite{HLOZ1997}, fluid dynamics \cite{AL2009}, and computational electromagnetics \cite{JH2003,CCCGW2011}. These problems are characterized by discontinuous coefficients across interfaces, which often induce reduced regularity of the solution, such as jumps or strong gradients near the interface. This lack of smoothness presents substantial challenges for the construction and analysis of accurate and robust numerical methods, particularly of high order.

A variety of numerical methods have been proposed to address elliptic interface problems. Broadly, these approaches can be divided into unfitted mesh methods and interface-fitted mesh methods. Unfitted mesh methods allow the computational mesh to be generated independently of the interface geometry, thereby avoiding complex mesh generation when interfaces are intricate or time dependent. One class of unfitted methods modifies basis functions locally near the interface to incorporate jump conditions directly into the approximation space. Representative examples include immersed interface methods \cite{ZL1998,JWCL2022, CCGL2022}, ghost fluid methods \cite{LS2003}, and hybridizable discontinuous Galerkin methods \cite{DWX2017,HCWX2020}.

Another class of unfitted methods enforces interface conditions weakly through penalty or stabilization terms. This category includes extended finite element methods \cite{XXW2020,CCW2022}, unfitted finite element methods \cite{HH2002}, cut finite element methods \cite{BCHLM2015}, and high-order hybridizable discontinuous Galerkin methods \cite{HNPK2013}. Although unfitted methods have demonstrated considerable flexibility, their analysis becomes increasingly intricate for complex interface geometries, and rigorous convergence results—especially for high-order schemes—remain challenging.

Interface-fitted mesh methods, by contrast, align the computational mesh with the interface and are well suited for problems involving complex geometries, irregular elements, or hanging nodes. Typical examples include discontinuous Galerkin methods \cite{HNPK2013,LW2022,WGC2022}, the matched interface and boundary method \cite{YW2007,ZZFW2006}, virtual element methods \cite{CWW2017}, and weak Galerkin finite element methods. The weak Galerkin method, originally introduced in \cite{ellip_JCAM2013} and further developed in \cite{wg1,wg2,wg3,wg4,wg5,wg6,wg7,wg8,wg9,wg10,wg11,wg12,wg13,wg14,wg15,wg16,wg17,wg18,wg19,wg20,wg21,itera,wz2023,wy3655}, is based on the use of weak derivatives and weak continuity, leading to schemes that are stable and highly adaptable to general meshes.

An important extension of the WG framework is the primal-dual weak Galerkin (PDWG) method \cite{pdwg1,pdwg2,pdwg3,pdwg4,pdwg5,pdwg6,pdwg7,pdwg8,pdwg9,pdwg10,pdwg11,pdwg12,pdwg13,pdwg14,pdwg15}. In the PDWG approach, numerical approximations are formulated as constrained minimization problems, with the governing equations enforced through weak constraints. The resulting Euler--Lagrange systems involve both primal and dual variables and possess favorable symmetry and stability properties.

In this work, we propose a simplified weak Galerkin finite element method for elliptic interface problems on both convex and nonconvex polygonal partitions. The method incorporates an intrinsic stabilization mechanism and does not require explicitly added stabilizer terms, which are commonly employed in classical WG formulations. In contrast to existing stabilizer-free WG methods \cite{ye}, the proposed scheme applies to nonconvex polygonal meshes and admits flexible choices of polynomial degrees. The essential analytical tool is the construction of suitable bubble functions, which enables stability without additional stabilization terms. Although this approach necessitates higher-degree polynomial spaces for computing discrete weak gradients, it facilitates the development of auto-stabilized WG methods on nonconvex elements.

The resulting scheme preserves the size and global sparsity pattern of the stiffness matrix, thereby significantly simplifying implementation relative to stabilizer-dependent WG methods. A rigorous error analysis demonstrates optimal-order convergence of the WG approximation in the discrete $H^1$ norm. These results confirm that the proposed auto-stabilized WG method provides an efficient and theoretically sound framework for elliptic interface problems on nonconvex polygonal meshes.

The remainder of this paper is organized as follows. Section~2 introduces the weak gradient operator and its discrete counterpart. Section~3 presents the WG discretization for elliptic interface problems on convex and nonconvex polygonal partitions. Section~4 establishes existence and uniqueness of the discrete solution. An error equation is derived in Section~5, followed by optimal-order error estimates in the discrete $H^1$ norm in Section~6. Numerical results validating the theoretical analysis are reported in Section~7.

Throughout this paper, standard Sobolev space notation is adopted as in \cite{GT1983}. Let $D\subset\mathbb{R}^2$ be a bounded domain with Lipschitz boundary $\partial D$. The symbols $(\cdot,\cdot)_{s,D}$, $|\cdot|_{s,D}$, and $\|\cdot\|_{s,D}$ denote the inner product, seminorm, and norm in $H^s(D)$ for integers $s\ge0$. When $s=0$, we write $(\cdot,\cdot)_D$ and $\|\cdot\|_D$. If $D=\Omega$, the subscript is omitted.

\section{Weak Gradient and Discrete Weak Gradient}\label{Section:Hessian}

In this section, we briefly review the definition of the weak gradient and its discrete counterpart, following \cite{wy3655, wang}.

Let $T$ be a polygonal element with boundary $\partial T$. A \emph{weak function} on $T$ is defined as
$v=\{v_0,v_b\}$, where $v_0\in L^2(T)$ and $v_b\in L^2(\partial T)$. The components $v_0$ and $v_b$ represent the values of $v$ in the interior of $T$ and on $\partial T$, respectively. In general, $v_b$ is treated as an independent variable and is not required to coincide with the trace of $v_0$. In the special case $v_b=v_0|_{\partial T}$, the weak function $v=\{v_0,v_b\}$ is uniquely determined by $v_0$ and may be identified with $v_0$.

Denote by $W(T)$ the space of all weak functions on $T$, namely,
\begin{equation}\label{2.1}
W(T)=\{v=\{v_0,v_b\}: v_0\in L^2(T),\; v_b\in L^2(\partial T)\}.
\end{equation}

The weak gradient, denoted by $\nabla_w$, is defined as a linear operator from $W(T)$ into the dual space of $[H^1(T)]^2$. For any $v\in W(T)$, the weak gradient $\nabla_w v$ is the bounded linear functional satisfying
\begin{equation}\label{2.3}
(\nabla_w v,\bvarphi)_T
=-(v_0,\nabla\cdot\bvarphi)_T
+\langle v_b,\bvarphi\cdot\bn\rangle_{\partial T},
\quad \forall \bvarphi\in[H^1(T)]^2,
\end{equation}
where $\bn$ denotes the unit outward normal vector to $\partial T$.

Let $r\ge0$ be an integer and let $P_r(T)$ denote the space of polynomials on $T$ of total degree at most $r$. The discrete weak gradient on $T$, denoted by $\nabla_{w,r,T}$, is a linear operator from $W(T)$ to $[P_r(T)]^2$. For any $v\in W(T)$, the discrete weak gradient $\nabla_{w,r,T}v$ is the unique polynomial vector in $[P_r(T)]^2$ satisfying
\begin{equation}\label{2.4}
(\nabla_{w,r,T} v,\bvarphi)_T
=-(v_0,\nabla\cdot\bvarphi)_T
+\langle v_b,\bvarphi\cdot\bn\rangle_{\partial T},
\quad \forall \bvarphi\in[P_r(T)]^2.
\end{equation}
If $v_0\in H^1(T)$ is sufficiently smooth, then applying integration by parts to the first term on the right-hand side of \eqref{2.4} yields
\begin{equation}\label{2.4new}
(\nabla_{w,r,T} v,\bvarphi)_T
=(\nabla v_0,\bvarphi)_T
+\langle v_b-v_0,\bvarphi\cdot\bn\rangle_{\partial T},
\quad \forall \bvarphi\in[P_r(T)]^2.
\end{equation}
\section{Weak Galerkin Algorithms}\label{Section:WGFEM}

Let ${\mathcal T}_h$ be a finite element partition of the domain $\Omega\subset\mathbb{R}^2$ into polygonal elements, assumed to be shape regular in the sense of \cite{wy3655}. Denote by ${\mathcal E}_h$ the set of all edges in ${\mathcal T}_h$ and by ${\mathcal E}_h^0={\mathcal E}_h\setminus\partial\Omega$ the set of all interior edges. For each $T\in{\mathcal T}_h$, let $h_T$ denote the diameter of $T$, and define the mesh size $h=\max_{T\in{\mathcal T}_h}h_T$.

Let $k\ge0$ and $q\ge0$ be integers satisfying $k\ge q$. For each element $T\in{\mathcal T}_h$, define the local weak finite element space by
\begin{equation}
V(k,q,T)=\{\{v_0,v_b\}: v_0\in P_k(T),\; v_b\in P_q(e),\ e\subset\partial T\}.
\end{equation}
By patching these local spaces together through a single-valued boundary component $v_b$ on each interior edge in ${\mathcal E}_h^0$, we obtain the global weak finite element space
\[
V_h=\big\{\{v_0,v_b\}:\ \{v_0,v_b\}|_T\in V(k,q,T),\ \forall T\in{\mathcal T}_h\big\}.
\]
Let $V_h^0$ denote the subspace of $V_h$ consisting of functions with vanishing boundary values on $\partial\Omega$, namely,
\[
V_h^0=\{v\in V_h:\ v_b=0 \ \text{on}\ \partial\Omega\}.
\]

For notational simplicity, we write $\nabla_w v$ to denote the discrete weak gradient computed elementwise by \eqref{2.4}; that is,
\[
(\nabla_w v)|_T=\nabla_{w,r,T}(v|_T),
\qquad \forall T\in{\mathcal T}_h.
\]

For any edge $e\subset\Gamma$ shared by two adjacent elements $T_1\subset\Omega_1$ and $T_2\subset\Omega_2$, define the jump of $v_b$ across $e$ by
\[
[[v_b]]_e
= v_b|_{\partial T_1\cap\Gamma}
- v_b|_{\partial T_2\cap\Gamma}.
\]

For $v,w\in V_h$, define the bilinear form
\[
a(v,w)=\sum_{T\in{\mathcal T}_h}({\color{black}{a}}\nabla_w v,\nabla_w w)_T.
\]
Let $Q_b$ denote the $L^2$ projection onto $P_q(e)$ on each edge $e$.

\begin{WG}\label{PDWG1}
A weak Galerkin scheme for the weak formulation \eqref{weak-formula} of the interface problem \eqref{model-1}--\eqref{model-4} seeks $u_h=\{u_0,u_b\}\in V_h$ such that
$u_b=Q_b g$ on $\partial\Omega\setminus\Gamma$ and
${[[u_b]]}_{\Gamma}=Q_b g_D$, and satisfies
\begin{equation}\label{wg}
\sum_{T\in{\mathcal T}_h}({\color{black}{a}}\nabla_w u_h,\nabla_w v)_T
=(f,v_0)+\sum_{e\subset\Gamma}\langle g_N,v_b\rangle_e,
\quad \forall v\in V_h^0.
\end{equation}
\end{WG}

\section{Solution Existence and Uniqueness} 
 
Recall that ${\mathcal T}_h$ is a shape-regular finite element partition of the domain $\Omega$. Therefore,  for any $T\in {\mathcal T}_h$ and $\phi\in H^1(T)$,
 the following trace inequality holds true \cite{wy3655}; i.e.,
\begin{equation}\label{tracein}
 \|\phi\|^2_{\partial T} \leq C(h_T^{-1}\|\phi\|_T^2+h_T \|\nabla \phi\|_T^2).
\end{equation}
If $\phi$ is a polynomial on the element $T\in {\mathcal T}_h$,  the following trace inequality holds true \cite{wy3655}; i.e.,
\begin{equation}\label{trace}
\|\phi\|^2_{\partial T} \leq Ch_T^{-1}\|\phi\|_T^2.
\end{equation}

For any $v=\{v_0, v_b\}\in V_h$, we define
the following discrete energy norm \begin{equation}\label{3norm}
\3bar v\3bar=\Big( \sum_{T\in {\mathcal T}_h} (a\nabla_w v, \nabla_w v)_T\Big)^{\frac{1}{2}},
\end{equation}
and the following discrete $H^1$ semi-norm 
\begin{equation}\label{disnorm}
\|v\|_{1, h}=\Big( \sum_{T\in {\mathcal T}_h} (a\nabla v_0, \nabla v_0)_T^2+h_T^{-1}\|v_0-v_b\|_{\partial T}^2\Big)^{\frac{1}{2}}.
\end{equation}
\begin{lemma}\cite{wang}\label{norm1}
 For $v=\{v_0, v_b\}\in V_h$, there exists a constant $C$ such that
 $$
 \|\nabla v_0\|_T\leq C\|\nabla_w v\|_T.
 $$
\end{lemma}

\begin{lemma}\cite{wang}\label{normeqva}   There exists two positive constants $C_1$ and $C_2$ such that for any $v=\{v_0, v_b\} \in V_h$, we have
 \begin{equation}\label{normeq}
 C_1\|v\|_{1, h}\leq \3bar v\3bar  \leq C_2\|v\|_{1, h}.
\end{equation}
\end{lemma}

\begin{theorem}
The  WG Algorithm  \ref{PDWG1} has  a unique solution. 
\end{theorem}
\begin{proof}
Assume $u_h^{(1)}\in V_h$ such that $u_b^{(1)}=Q_bg$ on $\pa\O\backslash \Gamma$, ${[[u_b^{(1)}]]}_{\Gamma}=Q_bg_D$  and $u_h^{(2)}\in V_h$ such that $u_b^{(2)}=Q_bg$ on $\pa\O\backslash \Gamma$, ${[[u_b^{(2)}]]}_{\Gamma}=Q_bg_D$ are two different solutions of the WG Algorithm  \ref{PDWG1}. Then, $\eta_h= u_h^{(1)}-u_h^{(2)}\in V_h^0$ satisfies
$$
\sum_{T\in {\mathcal T}_h}(a\nabla_w \eta_h, \nabla_w v)_T=0, \qquad \forall v\in V_h^0.
$$
Letting $v=\eta_h$  gives $\3bar \eta_h\3bar=0$. From \eqref{normeq} we have $\|\eta_h\|_{1,h}=0$, which yields
$\nabla \eta_0=0$ on each $T$ and $\eta_0=\eta_b$ on each $\partial T$. Using the fact that $\nabla \eta_0=0$ on each $T$ gives $\eta_0=C$ on each $T$. This, together with $\eta_0=\eta_b$ on each $\partial T$ and  $\eta_b=0$ on $\partial \Omega$,   gives $\eta_0\equiv 0$ and  further $\eta_b\equiv 0$  and $\eta_h\equiv 0$ in the domain $\Omega$. Therefore, we have $u_h^{(1)}\equiv u_h^{(2)}$. This completes the proof of this theorem.
\end{proof}
\section{Error Equations}

Let $\mathcal{T}_h$ be a partition of the domain $\Omega$. On each element $T \in \mathcal{T}_h$, let $Q_0$ be the $L^2$ projection onto $P_k(T)$. On each edge or face $e \subset \partial T$, let $Q_b$ be the $L^2$ projection operator onto $P_{q}(e)$. For any $w \in H^1(\Omega)$, we denote by $Q_h w$ the $L^2$ projection into the weak finite element space $V_h$, defined such that:
$$
(Q_h w)|_T := \{Q_0(w|_T), Q_b(w|_{\partial T})\}, \qquad \forall T \in \mathcal{T}_h.
$$
Let $r = 2N + k - 1$ for a nonconvex polygon with $N$ edges, and $r = N + k - 1$ for a convex polygon with $N$ edges. See \cite{wang} for details. We denote by $Q_r$ the $L^2$ projection operator onto the finite element space of piecewise polynomials of degree $r$.

\begin{lemma}\label{Lemma5.1} \cite{wang}
The weak gradient satisfies the following property:
\begin{equation}\label{pro}
\nabla_{w} u = Q_r \nabla u, \qquad \forall u \in H^1(T).
\end{equation}
\end{lemma} 

Let $u$ and $u_h$ be the exact solution of the elliptic interface problem \eqref{model-1}--\eqref{model-4} and its numerical approximation arising from the WG Algorithm \ref{PDWG1}, respectively. We define the error function, denoted by $e_h$, as:
\begin{equation}\label{error}
e_h = u - u_h.
\end{equation}

\begin{lemma}\label{errorequa}
The error function $e_h$ defined in \eqref{error} satisfies the following error equation:
\begin{equation}\label{erroreqn}
\sum_{T \in \mathcal{T}_h} (a \nabla_w e_h, \nabla_w v)_T = \sum_{T \in \mathcal{T}_h} \langle (I - Q_r) a \nabla u \cdot \bn, v_0 - v_b \rangle_{\partial T}, \qquad \forall v \in V_h^0.
\end{equation}
\end{lemma}

\begin{proof}
Using \eqref{pro} and setting $\bvarphi = Q_r (a \nabla u)$ in \eqref{2.4new}, we obtain:
\begin{equation}\label{54}
\begin{split}
&\sum_{T \in \mathcal{T}_h} (a \nabla_w u, \nabla_w v)_T \\
=&\sum_{T \in \mathcal{T}_h} (Q_r (a \nabla u), \nabla_w v)_T \\
=&\sum_{T \in \mathcal{T}_h}   (Q_r (a \nabla u), \nabla v_0)_T + \langle Q_r(a \nabla u) \cdot \bn, v_b - v_0 \rangle_{\partial T}  \\
=&\sum_{T \in \mathcal{T}_h}  (a \nabla u, \nabla v_0)_T + \langle Q_r(a \nabla u) \cdot \bn, v_b - v_0 \rangle_{\partial T}  \\
=&\sum_{T \in \mathcal{T}_h} (f, v_0)_T + \langle g_N, v_0 \rangle_{\Gamma} - \langle a \nabla u \cdot \bn, v_0 - v_b \rangle_{\partial T} \\
&\quad + \langle (I - Q_r)(a \nabla u \cdot \bn), v_0 - v_b \rangle_{\partial T} \\
=&\sum_{T \in \mathcal{T}_h} (f, v_0)_T + \langle g_N, v_0 \rangle_{\Gamma} - \langle g_N, v_0 - v_b \rangle_{\Gamma} \\
&\quad + \langle (I - Q_r)(a \nabla u \cdot \bn), v_0 - v_b \rangle_{\partial T} \\
=&\sum_{T \in \mathcal{T}_h} (f, v_0)_T + \sum_{T \in \mathcal{T}_h} \langle (I - Q_r)(a \nabla u \cdot \bn), v_0 - v_b \rangle_{\partial T} + \sum_{e \subset \Gamma} \langle g_N, v_b \rangle_e,
\end{split}
\end{equation}
where we used \eqref{weak-formula} and it follows from \eqref{model-4} that
$$
\sum_{T \in \mathcal{T}_h} \langle a \nabla u \cdot \bn, v_0 - v_b \rangle_{\partial T} = \langle [[a \nabla u \cdot \bn]]_{\Gamma}, v_0 - v_b \rangle_{\Gamma}=\langle g_N, v_0 - v_b \rangle_{\Gamma}.
$$

Subtracting \eqref{wg} from \eqref{54} yields:
$$
\sum_{T \in \mathcal{T}_h} (a \nabla_w e_h, \nabla_w v)_T = \sum_{T \in \mathcal{T}_h} \langle (I - Q_r) a \nabla u \cdot \bn, v_0 - v_b \rangle_{\partial T}.
$$
This completes the proof.
\end{proof}
 
\section{Error Estimates in $H^1$}

\begin{lemma}\cite{wy3655}
Let $\mathcal{T}_h$ be a finite element partition of the domain $\Omega$ satisfying the shape regularity assumption specified in \cite{wy3655}. For any $0\leq s \leq 1$, $0\leq n \leq k$, and $0\leq m \leq r$, there holds:
\begin{align}
\label{error1}
\sum_{T\in \mathcal{T}_h} h_T^{2s} \|\nabla u - Q_r \nabla u\|^2_{s,T} &\leq C h^{2m} \|u\|^2_{m+1}, \\
\label{error2}
\sum_{T\in \mathcal{T}_h} h_T^{2s} \|u - Q_0 u\|^2_{s,T} &\leq C h^{2n+2} \|u\|^2_{n+1}.
\end{align}
\end{lemma}

\begin{lemma}
Assume the exact solution $u$ of the elliptic interface problem \eqref{model-1}--\eqref{model-4} satisfies the regularity condition $u \in H^{k+1}(\Omega_i)$ for $i=1, 2$. There exists a constant $C$ such that the following estimate holds:
\begin{equation}\label{erroresti1}
\3bar u - Q_h u \3bar \leq C h^k (\|u\|_{k+1, \Omega_1} + \|u\|_{k+1, \Omega_2}).
\end{equation}
\end{lemma}

\begin{proof}
Using \eqref{2.4new}, the Cauchy-Schwarz inequality, the trace inequalities \eqref{tracein}--\eqref{trace}, and the estimate \eqref{error2} with $n=k$ and $s=0, 1$, we have:
\begin{equation*}
\begin{split}
&\quad \ \sum_{T\in \mathcal{T}_h} (a \nabla_w(u - Q_h u), \bv)_T \\
&= \sum_{T\in \mathcal{T}_h} (a \nabla(u - Q_0 u), \bv)_T + \langle a(Q_0 u - Q_b u), \bv \cdot \bn \rangle_{\partial T} \\
&\leq \Big(\sum_{T\in \mathcal{T}_h} \|\nabla(u - Q_0 u)\|^2_T\Big)^{\frac{1}{2}} \Big(\sum_{T\in \mathcal{T}_h} \|a^{\frac{1}{2}}\bv\|_T^2\Big)^{\frac{1}{2}} \\
&\quad + \Big(\sum_{T\in \mathcal{T}_h} \|Q_0 u - Q_b u\|_{\partial T}^2\Big)^{\frac{1}{2}} \Big(\sum_{T\in \mathcal{T}_h} \|a^{\frac{1}{2}}\bv\|_{\partial T}^2\Big)^{\frac{1}{2}} \\
&\leq \Big(\sum_{T\in \mathcal{T}_h} \|\nabla(u - Q_0 u)\|_T^2\Big)^{\frac{1}{2}} \Big(\sum_{T\in \mathcal{T}_h} \|a^{\frac{1}{2}}\bv\|_T^2\Big)^{\frac{1}{2}} \\
&\quad + \Big(\sum_{T\in \mathcal{T}_h} h_T^{-1} \|Q_0 u - u\|_{T}^2 + h_T \|Q_0 u - u\|_{1,T}^2\Big)^{\frac{1}{2}} \Big(\sum_{T\in \mathcal{T}_h} C h_T^{-1} \|a^{\frac{1}{2}}\bv\|_T^2\Big)^{\frac{1}{2}} \\
&\leq C h^k (\|u\|_{k+1, \Omega_1} + \|u\|_{k+1, \Omega_2}) \Big(\sum_{T\in \mathcal{T}_h} \|a^{\frac{1}{2}}\bv\|_T^2\Big)^{\frac{1}{2}},
\end{split}
\end{equation*}
for any $\bv \in [P_r(T)]^2$.

Letting $\bv = \nabla_w(u - Q_h u)$ gives:
$$
\sum_{T\in \mathcal{T}_h} (a \nabla_w(u - Q_h u), \nabla_w(u - Q_h u))_T \leq C h^k (\|u\|_{k+1, \Omega_1} + \|u\|_{k+1, \Omega_2}) \3bar u - Q_h u \3bar.
$$
This completes the proof of the lemma.
\end{proof}

\begin{theorem}
Assume the exact solution $u$ of the elliptic interface problem \eqref{model-1}--\eqref{model-4} satisfies the regularity condition $u \in H^{k+1}(\Omega_i)$ for $i=1, 2$. There exists a constant $C$ such that the following error estimate holds:
\begin{equation}\label{trinorm}
\3bar u - u_h \3bar \leq C h^k (\|u\|_{k+1, \Omega_1} + \|u\|_{k+1, \Omega_2}).
\end{equation}
\end{theorem}

\begin{proof}
Consider the right-hand side of the error equation \eqref{erroreqn}. Using the Cauchy-Schwarz inequality, the trace inequality \eqref{tracein}, the estimate \eqref{error1} with $m=k$ and $s=0,1$, and the norm equivalence \eqref{normeq}, we obtain:
\begin{equation}\label{erroreqn1}
\begin{split}
&\Big| \sum_{T\in \mathcal{T}_h} \langle (I - Q_r) a \nabla u \cdot \bn, v_0 - v_b \rangle_{\partial T} \Big| \\
 \leq & C \Big( \sum_{T\in \mathcal{T}_h} \|(I - Q_r) a \nabla u \cdot \bn\|^2_T + h_T^2 \|\nabla((I - Q_r) a \nabla u \cdot \bn)\|^2_T \Big)^{\frac{1}{2}} \\
 \quad & \cdot \Big( \sum_{T\in \mathcal{T}_h} h_T^{-1} \|v_0 - v_b\|^2_{\partial T} \Big)^{\frac{1}{2}} \\
 \leq  & C h^k (\|u\|_{k+1, \Omega_1} + \|u\|_{k+1, \Omega_2}) \|v\|_{1,h} \\
 \leq & C h^k (\|u\|_{k+1, \Omega_1} + \|u\|_{k+1, \Omega_2}) \3bar v \3bar.
\end{split}
\end{equation}
Substituting \eqref{erroreqn1} into \eqref{erroreqn} gives:
\begin{equation}\label{err}
(a \nabla_w e_h, \nabla_w v) \leq C h^k (\|u\|_{k+1, \Omega_1} + \|u\|_{k+1, \Omega_2}) \3bar v \3bar.
\end{equation}

Using the Cauchy-Schwarz inequality, letting $v = Q_h u - u_h$ in \eqref{err}, and applying estimate \eqref{erroresti1}, we have:
\begin{equation*}
\begin{split}
&\quad \ \3bar u - u_h \3bar^2 \\
&= \sum_{T\in \mathcal{T}_h} (a \nabla_w (u - u_h), \nabla_w (u - Q_h u))_T + (a \nabla_w (u - u_h), \nabla_w (Q_h u - u_h))_T \\
&\leq \Big( \sum_{T\in \mathcal{T}_h} \|a^{\frac{1}{2}} \nabla_w (u - u_h)\|^2_T \Big)^{\frac{1}{2}} \Big( \sum_{T\in \mathcal{T}_h} \|a^{\frac{1}{2}} \nabla_w (u - Q_h u)\|^2_T \Big)^{\frac{1}{2}} \\
&\quad + \sum_{T\in \mathcal{T}_h} (a \nabla_w (u - u_h), \nabla_w (Q_h u - u_h))_T \\
&\leq \3bar u - u_h \3bar \3bar u - Q_h u \3bar + C h^k (\|u\|_{k+1, \Omega_1} + \|u\|_{k+1, \Omega_2})\\
  &\quad \ \cdot \Big( \sum_{T\in \mathcal{T}_h} \|a^{\frac{1}{2}} \nabla_w (Q_h u - u_h)\|^2_T \Big)^{\frac{1}{2}} \\
&\leq C \3bar u - u_h \3bar h^k (\|u\|_{k+1, \Omega_1} + \|u\|_{k+1, \Omega_2}) 
   + C h^k (\|u\|_{k+1, \Omega_1} + \|u\|_{k+1, \Omega_2}) \\
&\quad \ \cdot\Big( \sum_{T\in \mathcal{T}_h} \left( \|a^{\frac{1}{2}} \nabla_w (Q_h u - u)\|^2_T + \|a^{\frac{1}{2}} \nabla_w (u - u_h)\|^2_T \right) \Big)^{\frac{1}{2}} \\
&\leq C \3bar u - u_h \3bar h^k (\|u\|_{k+1, \Omega_1} + \|u\|_{k+1, \Omega_2}) 
  + C h^{2k} (\|u\|_{k+1, \Omega_1} + \|u\|_{k+1, \Omega_2})^2\\
&\quad \  + C h^k (\|u\|_{k+1, \Omega_1} + \|u\|_{k+1, \Omega_2}) \3bar u - u_h \3bar.
\end{split}
\end{equation*}
Simplifying the inequality gives:
\begin{equation*}
\3bar u - u_h \3bar \leq C h^k (\|u\|_{k+1, \Omega_1} + \|u\|_{k+1, \Omega_2}).
\end{equation*}
This completes the proof of the theorem.
\end{proof}

\section{Numerical experiments}
In the first test,  we solve the interface problem \eqref{model-1}--\eqref{model-4} 
   on a square domain $\Omega= (-1,1)\times (-1,1)$, where $g=0$, $g_D=0$, $g_N=0$,
\a{ a=\begin{cases} \lambda, \quad & \tx{if} \ x\le 0, \\
                     1 , \quad & \tx{if} \ 0<x, \end{cases} }
and
\a{ f=\begin{cases} -2 x^2 + (2\lambda - 2)x - 2y^2 + 2\lambda + 2, \quad & \tx{if} \ x\le 0, \\
                    (-2x^2 - 2y^2 + 2x + 2)\lambda - 2x + 2 , \quad & \tx{if} \ 0<x. \end{cases} }
The exact solution is 
\an{\label{u-1} u=\begin{cases} (1 + x)(-y^2 + 1)(1- \dfrac x \lambda), \quad & \tx{if} \ x\le 0, \\
                    (1 - x)(-y^2 + 1)(1+ \lambda x ), \quad & \tx{if} \ 0<x. \end{cases} }
We show this solution on the level 3 grid ($G_3$, shown in Figure \ref{f21}) in
      Figure \ref{fs-1}.   
This solution reflects the natural of the interface problem, i.e., $u\ne 0$ and $\partial_{\b n}\ne0$
  on the interface $\Gamma$.

\begin{figure}[H]
 \begin{center}\setlength\unitlength{1.0pt}
\begin{picture}(300,290)(0,0) 
  \put(0,-105){\includegraphics[width=300pt]{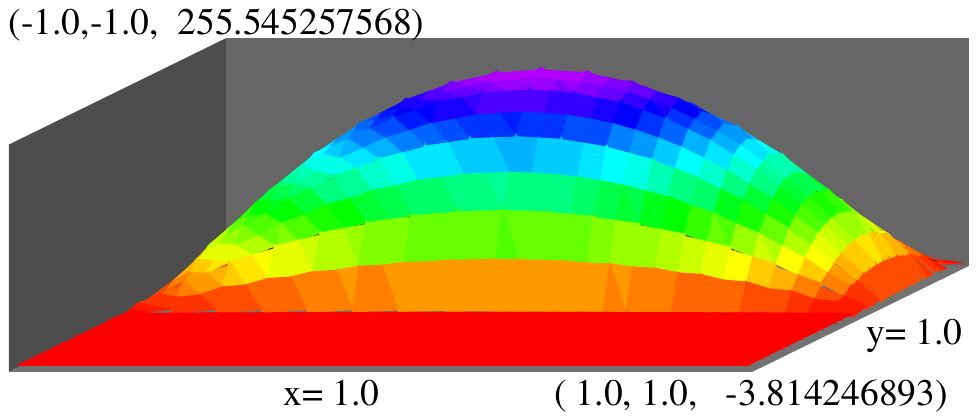}}
  \put(0,-205){\includegraphics[width=300pt]{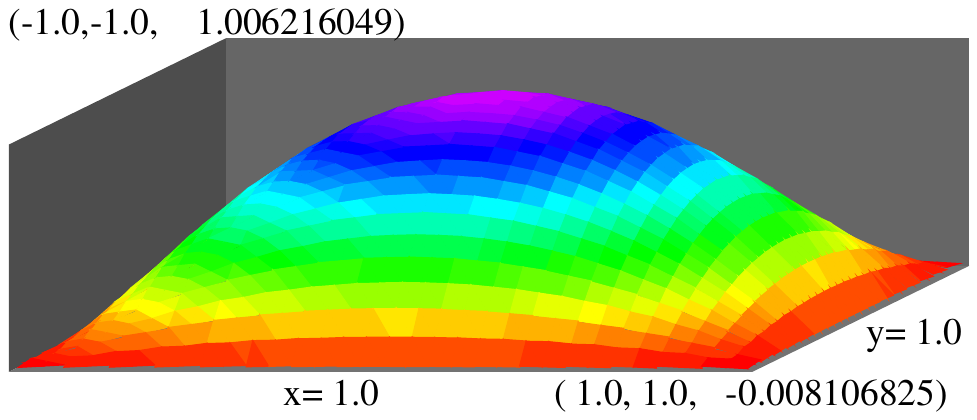}}
  \put(0,-305){\includegraphics[width=300pt]{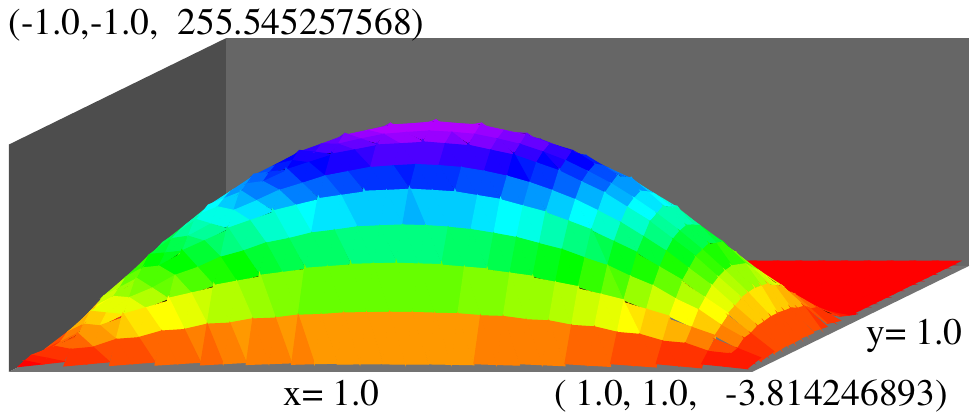}}  
 \end{picture}\end{center}
\caption{The exact solution $u$ in \eqref{u-1} for $\lambda=10^{-3}$(top), $\lambda=1$ and
  $\lambda=10^3$. }\label{fs-1}
\end{figure}

The solution in \eqref{u-1} is approximated by the weak Galerkin finite element
   $P_k$-$P_k$/$P_{k+2}$ (for $\{u_0, u_b\}$/$\nabla_w$), $k= 1,2,3$, on nonconvex polygonal
    grids shown in Figure \ref{f21}.
The errors and the computed orders of convergence are listed in Tables \ref{t1}--\ref{t3}.
The optimal order of convergence is achieved in every case.
    
\begin{figure}[H]
 \begin{center}\setlength\unitlength{1.0pt}
\begin{picture}(360,120)(0,0)
  \put(15,108){$G_1$:} \put(125,108){$G_2$:} \put(235,108){$G_3$:} 
  \put(0,-420){\includegraphics[width=380pt]{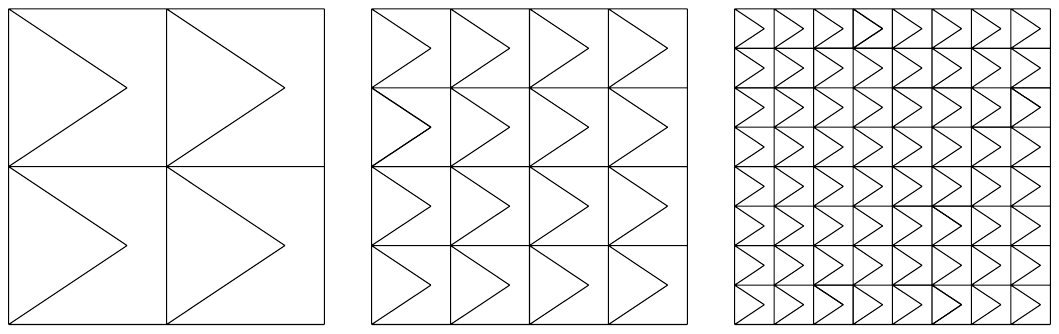}}  
 \end{picture}\end{center}
\caption{The nonconvex polygonal grids for the computation in Tables \ref{t1}--\ref{t3}. }\label{f21}
\end{figure}

  \vskip -.5cm
  \begin{table}[H]
  \caption{By the $P_1$-$P_1$/$P_3$ element for \eqref{u-1}  on Figure \ref{f21} grids.} \label{t1}
\begin{center}  
   \begin{tabular}{c|rr|rr}  
 \hline 
$G_i$ &  $ \|Q_h  u -   u_h \| $ & $O(h^r)$ &  $ \| a\nabla( Q_h u- u_h )\|_0 $ & $O(h^r)$  \\ \hline 
&\multicolumn{4}{c}{$\lambda=10^{-3}$ }\\
 \hline 
 4&    0.100E+02 &  1.7&    0.428E+01 &  1.1\\
 5&    0.265E+01 &  1.9&    0.208E+01 &  1.0\\
 6&    0.671E+00 &  2.0&    0.103E+01 &  1.0\\  \hline  
 &\multicolumn{4}{c}{$\lambda=1$ }\\
  \hline  
 4&    0.340E-01 &  1.8&    0.309E+00 &  1.1\\
 5&    0.878E-02 &  2.0&    0.150E+00 &  1.0\\
 6&    0.221E-02 &  2.0&    0.744E-01 &  1.0\\  \hline  
 &\multicolumn{4}{c}{$\lambda=10^3$ }\\
  \hline  
 4&    0.100E+02 &  1.7&    0.135E+03 &  1.1\\
 5&    0.265E+01 &  1.9&    0.659E+02 &  1.0\\
 6&    0.671E+00 &  2.0&    0.327E+02 &  1.0\\
  \hline  
\end{tabular} \end{center}  \end{table}

  \vskip -.5cm
  \begin{table}[H]
  \caption{By the $P_2$-$P_2$/$P_4$ element for \eqref{u-1} on Figure \ref{f21} grids.} \label{t2}
\begin{center}  
   \begin{tabular}{c|rr|rr}  
 \hline 
$G_i$ &  $ \|Q_h  u -   u_h \| $ & $O(h^r)$ &  $ \| a\nabla( Q_h u- u_h )\|_0 $ & $O(h^r)$  \\ \hline 
&\multicolumn{4}{c}{$\lambda=10^{-3}$ }\\
 \hline 
 3&    0.535E+00 &  3.6&    0.107E+01 &  2.0\\
 4&    0.450E-01 &  3.6&    0.266E+00 &  2.0\\
 5&    0.458E-02 &  3.3&    0.666E-01 &  2.0\\ \hline  
 &\multicolumn{4}{c}{$\lambda=1$ }\\
  \hline  
 3&    0.680E-03 &  3.5&    0.509E-01 &  2.0\\
 4&    0.633E-04 &  3.4&    0.126E-01 &  2.0\\
 5&    0.698E-05 &  3.2&    0.316E-02 &  2.0\\  \hline  
 &\multicolumn{4}{c}{$\lambda=10^3$ }\\
  \hline  
 3&    0.535E+00 &  3.6&    0.338E+02 &  2.0\\
 4&    0.450E-01 &  3.6&    0.841E+01 &  2.0\\
 5&    0.458E-02 &  3.3&    0.211E+01 &  2.0\\
  \hline  
\end{tabular} \end{center}  \end{table}

  \vskip -.5cm
  \begin{table}[H]
  \caption{By the $P_3$-$P_3$/$P_5$ element for \eqref{u-1}  on Figure \ref{f21} grids.} \label{t3}
\begin{center}  
   \begin{tabular}{c|rr|rr}  
 \hline 
$G_i$ &  $ \|Q_h  u -   u_h \| $ & $O(h^r)$ &  $ \| a\nabla( Q_h u- u_h )\|_0 $ & $O(h^r)$  \\ \hline 
&\multicolumn{4}{c}{$\lambda=10^{-3}$ }\\
 \hline 
 2&    0.158E+00 &  4.4&    0.196E+00 &  2.9\\
 3&    0.848E-02 &  4.2&    0.255E-01 &  2.9\\
 4&    0.485E-03 &  4.1&    0.323E-02 &  3.0\\\hline  
 &\multicolumn{4}{c}{$\lambda=1$ }\\
  \hline  
 2&    0.204E-03 &  4.3&    0.898E-02 &  2.9\\
 3&    0.114E-04 &  4.2&    0.115E-02 &  3.0\\
 4&    0.666E-06 &  4.1&    0.145E-03 &  3.0\\ \hline  
 &\multicolumn{4}{c}{$\lambda=10^3$ }\\
  \hline  
 2&    0.158E+00 &  4.4&    0.621E+01 &  2.9\\
 3&    0.848E-02 &  4.2&    0.805E+00 &  2.9\\
 4&    0.485E-03 &  4.1&    0.102E+00 &  3.0\\
  \hline  
\end{tabular} \end{center}  \end{table}

We compute the solution $u$ in \eqref{u-1} again,  by the weak Galerkin finite element
   $P_k$-$P_k$/$P_{k+3}$ (for $\{u_0, u_b\}$/$\nabla_w$), $k= 1,2,3$, on more nonconvex polygonal
    grids shown in Figure \ref{f31}.
The errors and the computed orders of convergence are listed in Tables \ref{t4}--\ref{t6}.
The optimal order of convergence is achieved in all cases.
The results are slightly better, in a few cases,
   than those in Tables \ref{t1}--\ref{t3}, though the grids here are worse.
We would think it is due to a much larger $u_b$ space on (more) edges.
    
\begin{figure}[H]
 \begin{center}\setlength\unitlength{1.0pt}
\begin{picture}(360,120)(0,0)
  \put(15,108){$G_1$:} \put(125,108){$G_2$:} \put(235,108){$G_3$:} 
  \put(0,-420){\includegraphics[width=380pt]{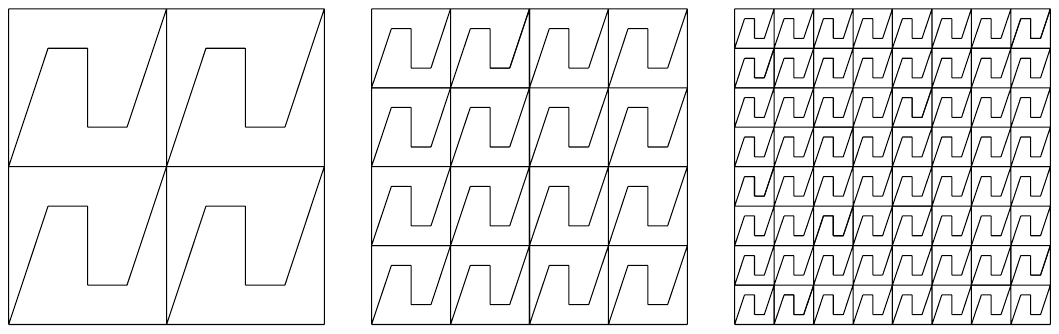}}  
 \end{picture}\end{center}
\caption{The nonconvex polygonal grids for the computation in Tables \ref{t4}--\ref{t6}. } \label{f31}
\end{figure}

  \vskip -.5cm
  \begin{table}[H]
  \caption{By the $P_1$-$P_1$/$P_4$ element for \eqref{u-1}  on Figure \ref{f31} grids.} \label{t4}
\begin{center}  
   \begin{tabular}{c|rr|rr}  
 \hline 
$G_i$ &  $ \|Q_h  u -   u_h \| $ & $O(h^r)$ &  $ \| a\nabla( Q_h u- u_h )\|_0 $ & $O(h^r)$  \\ \hline 
&\multicolumn{4}{c}{$\lambda=10^{-3}$ }\\
 \hline 
 4&    0.484E+01 &  1.9&    0.617E+01 &  1.0\\
 5&    0.124E+01 &  2.0&    0.309E+01 &  1.0\\
 6&    0.311E+00 &  2.0&    0.154E+01 &  1.0\\  \hline  
 &\multicolumn{4}{c}{$\lambda=1$ }\\
  \hline  
 4&    0.145E-01 &  1.9&    0.471E+00 &  1.0\\
 5&    0.366E-02 &  2.0&    0.235E+00 &  1.0\\
 6&    0.918E-03 &  2.0&    0.118E+00 &  1.0\\  \hline  
 &\multicolumn{4}{c}{$\lambda=10^3$ }\\
  \hline  
 4&    0.484E+01 &  1.9&    0.195E+03 &  1.0\\
 5&    0.124E+01 &  2.0&    0.976E+02 &  1.0\\
 6&    0.311E+00 &  2.0&    0.488E+02 &  1.0\\
  \hline  
\end{tabular} \end{center}  \end{table}

  \vskip -.5cm
  \begin{table}[H]
  \caption{By the $P_2$-$P_2$/$P_5$ element for \eqref{u-1} on Figure \ref{f31} grids.} \label{t5}
\begin{center}  
   \begin{tabular}{c|rr|rr}  
 \hline 
$G_i$ &  $ \|Q_h  u -   u_h \| $ & $O(h^r)$ &  $ \| a\nabla( Q_h u- u_h )\|_0 $ & $O(h^r)$  \\ \hline 
&\multicolumn{4}{c}{$\lambda=10^{-3}$ }\\
 \hline 
 3&    0.308E+00 &  3.3&    0.965E+00 &  1.9\\
 4&    0.338E-01 &  3.2&    0.244E+00 &  2.0\\
 5&    0.401E-02 &  3.1&    0.611E-01 &  2.0\\ \hline  
 &\multicolumn{4}{c}{$\lambda=1$ }\\
  \hline  
 3&    0.569E-03 &  3.3&    0.613E-01 &  2.0\\
 4&    0.652E-04 &  3.1&    0.154E-01 &  2.0\\
 5&    0.789E-05 &  3.0&    0.385E-02 &  2.0\\  \hline  
 &\multicolumn{4}{c}{$\lambda=10^3$ }\\
  \hline  
 3&    0.308E+00 &  3.3&    0.305E+02 &  1.9\\
 4&    0.338E-01 &  3.2&    0.771E+01 &  2.0\\
 5&    0.401E-02 &  3.1&    0.193E+01 &  2.0\\
  \hline  
\end{tabular} \end{center}  \end{table}

  \vskip -.5cm
  \begin{table}[H]
  \caption{By the $P_3$-$P_3$/$P_6$ element for \eqref{u-1}  on Figure \ref{f31} grids.} \label{t6}
\begin{center}  
   \begin{tabular}{c|rr|rr}  
 \hline 
$G_i$ &  $ \|Q_h  u -   u_h \| $ & $O(h^r)$ &  $ \| a\nabla( Q_h u- u_h )\|_0 $ & $O(h^r)$  \\ \hline 
&\multicolumn{4}{c}{$\lambda=10^{-3}$ }\\
 \hline 
 2&    0.149E+00 &  4.4&    0.310E+00 &  3.0\\
 3&    0.813E-02 &  4.2&    0.393E-01 &  3.0\\
 4&    0.471E-03 &  4.1&    0.495E-02 &  3.0\\ \hline  
 &\multicolumn{4}{c}{$\lambda=1$ }\\
  \hline  
 2&    0.198E-03 &  4.3&    0.140E-01 &  3.0\\
 3&    0.111E-04 &  4.2&    0.177E-02 &  3.0\\
 4&    0.686E-06 &  4.0&    0.222E-03 &  3.0\\ \hline  
 &\multicolumn{4}{c}{$\lambda=10^3$ }\\
  \hline  
 2&    0.149E+00 &  4.4&    0.981E+01 &  3.0\\
 3&    0.813E-02 &  4.2&    0.124E+01 &  3.0\\
 4&    0.471E-03 &  4.1&    0.157E+00 &  3.0\\
  \hline  
\end{tabular} \end{center}  \end{table}

In the second test,  we solve the interface problem \eqref{model-1}--\eqref{model-4} also
   on a square domain $\Omega= (-1,1)\times (-1,1)$, where $g=0$, $g_D=0$, $g_N=0$,
\a{ a=\begin{cases} 1, \quad & \tx{if} \ (x,y)\in(-\frac 13,\frac 13)^2=:\Omega_1, \\
                     \lambda , \quad & \tx{if} \ (x,y)\in \Omega\setminus \Omega_1, \end{cases} }
and
\a{ f=-4 (243 x^4 y^2 + 243 x^2 y^4 - 45 x^4 - 540 x^2 y^2
           - 45 y^4 + 77 x^2 + 77 y^2 - 10). }
The exact solution is  
\an{\label{u-2} u=\begin{cases} (-x^2 + 1) (-y^2 + 1) (-9 x^2 + 1) (-9 y^2 + 1), 
\quad & \tx{if} \ (x,y)\in \Omega_1, \\
                      \lambda^{-1} (-x^2 + 1) (-y^2 + 1) (-9 x^2 + 1) (-9 y^2 + 1)
                       , \quad & \tx{if} \ (x,y)\in \Omega\setminus \Omega_1, \end{cases} }  
                       where $\Omega_1$ is defined one equation above.                   
We show this solution on level 3 grid ($G_3$, shown in Figure \ref{f25}) in
      Figure \ref{fs-2}.

\begin{figure}[H]
 \begin{center}\setlength\unitlength{1.0pt}
\begin{picture}(300,290)(0,0) 
  \put(0,-105){\includegraphics[width=300pt]{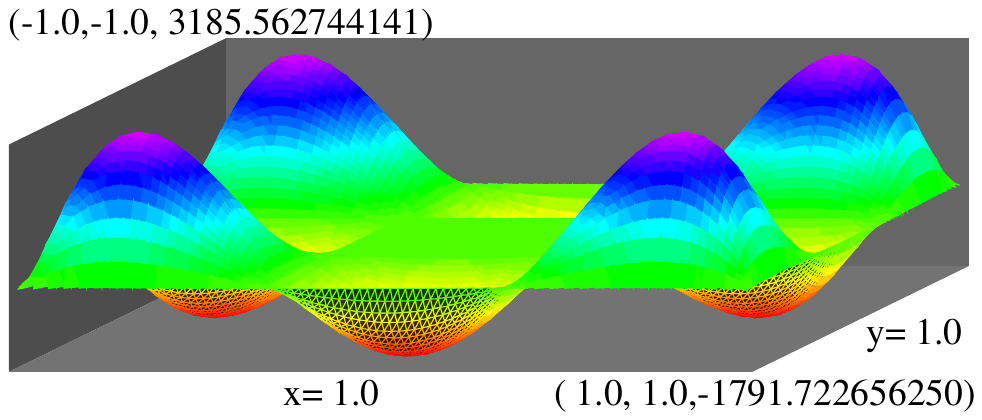}}
  \put(0,-205){\includegraphics[width=300pt]{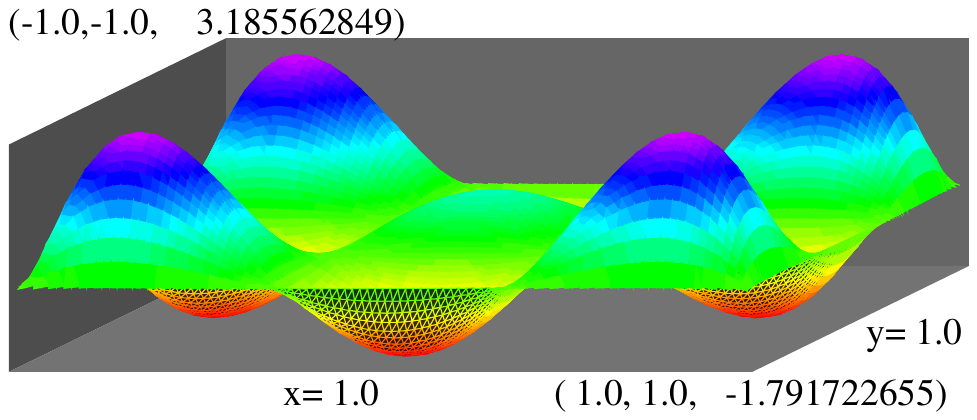}}
  \put(0,-305){\includegraphics[width=300pt]{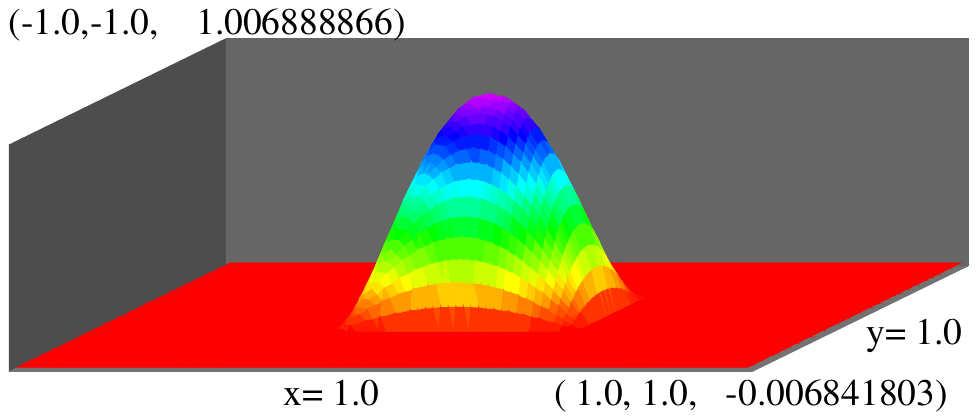}}  
 \end{picture}\end{center}
\caption{The exact solution $u$ in \eqref{u-2} for $\lambda=10^{-3}$(top), $\lambda=1$ and
  $\lambda=10^3$. }\label{fs-2}
\end{figure}

The solution in \eqref{u-2} is approximated by the weak Galerkin finite element
   $P_k$-$P_k$/$P_{k+2}$ (for $\{u_0, u_b\}$/$\nabla_w$), $k= 1,2,3$, on nonconvex polygonal
    grids shown in Figure \ref{f25}.
The errors and the computed orders of convergence are listed in Tables \ref{t7}--\ref{t9}.
The optimal order of convergence is achieved in every case.
    
\begin{figure}[H]
 \begin{center}\setlength\unitlength{1.0pt}
\begin{picture}(360,120)(0,0)
  \put(15,108){$G_1$:} \put(125,108){$G_2$:} \put(235,108){$G_3$:} 
  \put(0,-420){\includegraphics[width=380pt]{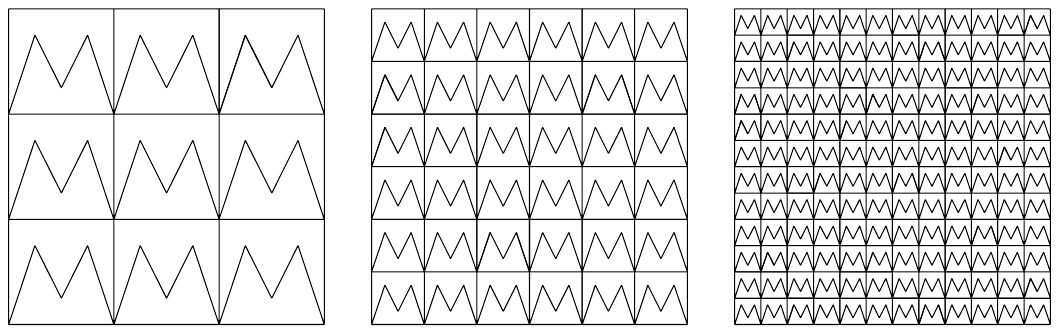}}  
 \end{picture}\end{center}
\caption{The nonconvex polygonal grids for the computation in Tables \ref{t7}--\ref{t9}. }\label{f25}
\end{figure}

  \vskip -.5cm
  \begin{table}[H]
  \caption{By the $P_1$-$P_1$/$P_3$ element for \eqref{u-2}  on Figure \ref{f25} grids.} \label{t7}
\begin{center}  
   \begin{tabular}{c|rr|rr}  
 \hline 
$G_i$ &  $ \|Q_h  u -   u_h \| $ & $O(h^r)$ &  $ \| a\nabla( Q_h u- u_h )\|_0 $ & $O(h^r)$  \\ \hline 
&\multicolumn{4}{c}{$\lambda=10^{-3}$ }\\
 \hline 
 3&    0.270E+03 &  1.3&    0.312E+03 &  1.2\\
 4&    0.782E+02 &  1.8&    0.158E+03 &  1.0\\
 5&    0.203E+02 &  1.9&    0.790E+02 &  1.0\\  \hline  
 &\multicolumn{4}{c}{$\lambda=1$ }\\
  \hline  
 3&    0.273E+00 &  1.3&    0.990E+01 &  1.0\\
 4&    0.790E-01 &  1.8&    0.501E+01 &  1.0\\
 5&    0.205E-01 &  1.9&    0.252E+01 &  1.0\\  \hline  
 &\multicolumn{4}{c}{$\lambda=10^3$ }\\
  \hline  
 3&    0.359E-01 &  1.4&    0.152E+01 &  2.8\\
 4&    0.100E-01 &  1.8&    0.606E+00 &  1.3\\
 5&    0.256E-02 &  2.0&    0.300E+00 &  1.0\\
  \hline  
\end{tabular} \end{center}  \end{table}

  \vskip -.5cm
  \begin{table}[H]
  \caption{By the $P_2$-$P_2$/$P_4$ element for \eqref{u-2} on Figure \ref{f25} grids.} \label{t8}
\begin{center}  
   \begin{tabular}{c|rr|rr}  
 \hline 
$G_i$ &  $ \|Q_h  u -   u_h \| $ & $O(h^r)$ &  $ \| a\nabla( Q_h u- u_h )\|_0 $ & $O(h^r)$  \\ \hline 
&\multicolumn{4}{c}{$\lambda=10^{-3}$ }\\
 \hline 
 2&    0.133E+03 &  2.5&    0.259E+03 &  4.4\\
 3&    0.168E+02 &  3.0&    0.515E+02 &  2.3\\
 4&    0.190E+01 &  3.1&    0.132E+02 &  2.0\\ \hline  
 &\multicolumn{4}{c}{$\lambda=1$ }\\
  \hline  
 2&    0.134E+00 &  2.5&    0.612E+01 &  2.5\\
 3&    0.168E-01 &  3.0&    0.162E+01 &  1.9\\
 4&    0.190E-02 &  3.1&    0.419E+00 &  2.0\\  \hline  
 &\multicolumn{4}{c}{$\lambda=10^3$ }\\
  \hline  
 2&    0.132E-01 &  2.2&    0.310E+01 &  1.2\\
 3&    0.151E-02 &  3.1&    0.238E+00 &  3.7\\
 4&    0.176E-03 &  3.1&    0.446E-01 &  2.4\\
  \hline  
\end{tabular} \end{center}  \end{table}

  \vskip -.5cm
  \begin{table}[H]
  \caption{By the $P_3$-$P_3$/$P_5$ element for \eqref{u-2}  on Figure \ref{f25} grids.} \label{t9}
\begin{center}  
   \begin{tabular}{c|rr|rr}  
 \hline 
$G_i$ &  $ \|Q_h  u -   u_h \| $ & $O(h^r)$ &  $ \| a\nabla( Q_h u- u_h )\|_0 $ & $O(h^r)$  \\ \hline 
&\multicolumn{4}{c}{$\lambda=10^{-3}$ }\\
 \hline 
 2&    0.169E+02 &  4.1&    0.455E+02 &  3.5\\
 3&    0.101E+01 &  4.1&    0.480E+01 &  3.2\\
 4&    0.592E-01 &  4.1&    0.607E+00 &  3.0\\ \hline  
 &\multicolumn{4}{c}{$\lambda=1$ }\\
  \hline  
 2&    0.170E-01 &  4.1&    0.123E+01 &  3.6\\
 3&    0.101E-02 &  4.1&    0.151E+00 &  3.0\\
 4&    0.593E-04 &  4.1&    0.193E-01 &  3.0\\  \hline  
 &\multicolumn{4}{c}{$\lambda=10^3$ }\\
  \hline   
  \hline  
\end{tabular} \end{center}  \end{table}

We compute the solution $u$ in \eqref{u-2} again,  by the weak Galerkin finite element
   $P_k$-$P_k$/$P_{k+3}$ (for $\{u_0, u_b\}$/$\nabla_w$), $k= 1,2,3$, on more nonconvex polygonal
    grids shown in Figure \ref{f23}.
The errors and the computed orders of convergence are listed in Tables \ref{t10}--\ref{t12}.
The optimal order of convergence is achieved in all cases.
The results are similar to those in Tables \ref{t7}--\ref{t9}, though the grids here are worse.

\begin{figure}[H]
 \begin{center}\setlength\unitlength{1.0pt}
\begin{picture}(380,120)(0,0)
  \put(15,108){$G_1$:} \put(125,108){$G_2$:} \put(235,108){$G_3$:} 
  \put(0,-420){\includegraphics[width=380pt]{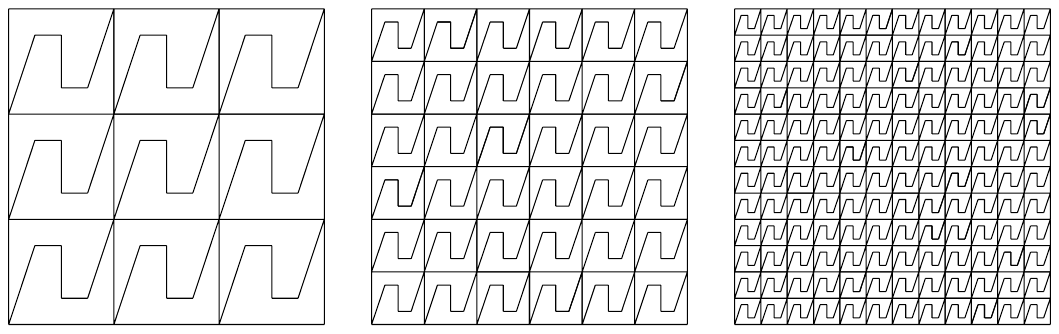}}  
 \end{picture}\end{center}
\caption{The nonconvex polygonal grids for the computation in Tables \ref{t10}--\ref{t12}. } \label{f23}
\end{figure}

  \vskip -.5cm
  \begin{table}[H]
  \caption{By the $P_1$-$P_1$/$P_4$ element for \eqref{u-2}  on Figure \ref{f23} grids.} \label{t10}
\begin{center}  
   \begin{tabular}{c|rr|rr}  
 \hline 
$G_i$ &  $ \|Q_h  u -   u_h \| $ & $O(h^r)$ &  $ \| a\nabla( Q_h u- u_h )\|_0 $ & $O(h^r)$  \\ \hline 
&\multicolumn{4}{c}{$\lambda=10^{-3}$ }\\
 \hline 
 3&    0.444E+03 &  1.1&    0.484E+03 &  1.7\\
 4&    0.137E+03 &  1.7&    0.237E+03 &  1.0\\
 5&    0.364E+02 &  1.9&    0.118E+03 &  1.0\\  \hline  
 &\multicolumn{4}{c}{$\lambda=1$ }\\
  \hline  
 3&    0.456E+00 &  1.1&    0.152E+02 &  1.4\\
 4&    0.140E+00 &  1.7&    0.753E+01 &  1.0\\
 5&    0.372E-01 &  1.9&    0.377E+01 &  1.0\\ \hline  
 &\multicolumn{4}{c}{$\lambda=10^3$ }\\
  \hline  
 3&    0.604E-01 &  1.3&    0.305E+01 &  3.2\\
 4&    0.178E-01 &  1.8&    0.954E+00 &  1.7\\
 5&    0.465E-02 &  1.9&    0.461E+00 &  1.0\\
  \hline  
\end{tabular} \end{center}  \end{table}

  \vskip -.5cm
  \begin{table}[H]
  \caption{By the $P_2$-$P_2$/$P_5$ element for \eqref{u-2} on Figure \ref{f23} grids.} \label{t11}
\begin{center}  
   \begin{tabular}{c|rr|rr}  
 \hline 
$G_i$ &  $ \|Q_h  u -   u_h \| $ & $O(h^r)$ &  $ \| a\nabla( Q_h u- u_h )\|_0 $ & $O(h^r)$  \\ \hline 
&\multicolumn{4}{c}{$\lambda=10^{-3}$ }\\
 \hline 
 3&    0.196E+02 &  3.5&    0.759E+02 &  3.1\\
 4&    0.199E+01 &  3.3&    0.184E+02 &  2.0\\
 5&    0.223E+00 &  3.2&    0.465E+01 &  2.0\\  \hline  
 &\multicolumn{4}{c}{$\lambda=1$ }\\
  \hline  
 3&    0.196E-01 &  3.5&    0.233E+01 &  2.5\\
 4&    0.199E-02 &  3.3&    0.584E+00 &  2.0\\
 5&    0.224E-03 &  3.2&    0.148E+00 &  2.0\\  \hline  
 &\multicolumn{4}{c}{$\lambda=10^3$ }\\
  \hline  
 3&    0.152E-02 &  3.3&    0.501E+00 &  3.9\\
 4&    0.171E-03 &  3.2&    0.663E-01 &  2.9\\
 5&    0.205E-04 &  3.1&    0.156E-01 &  2.1\\
  \hline  
\end{tabular} \end{center}  \end{table}

  \vskip -.5cm
  \begin{table}[H]
  \caption{By the $P_3$-$P_3$/$P_6$ element for \eqref{u-2}  on Figure \ref{f23} grids.} \label{t12}
\begin{center}  
   \begin{tabular}{c|rr|rr}  
 \hline 
$G_i$ &  $ \|Q_h  u -   u_h \| $ & $O(h^r)$ &  $ \| a\nabla( Q_h u- u_h )\|_0 $ & $O(h^r)$  \\ \hline 
&\multicolumn{4}{c}{$\lambda=10^{-3}$ }\\
 \hline 
 2&    0.346E+02 &  5.8&    0.120E+03 &  5.6\\
 3&    0.117E+01 &  4.9&    0.661E+01 &  4.2\\
 4&    0.611E-01 &  4.3&    0.770E+00 &  3.1\\ \hline  
 &\multicolumn{4}{c}{$\lambda=1$ }\\
  \hline  
 2&    0.345E-01 &  5.8&    0.291E+01 &  4.9\\
 3&    0.117E-02 &  4.9&    0.202E+00 &  3.8\\
 4&    0.612E-04 &  4.3&    0.244E-01 &  3.0\\ \hline  
 &\multicolumn{4}{c}{$\lambda=10^3$ }\\
  \hline  
 2&    0.268E-02 &  5.8&    0.346E+01 &  5.4\\
 3&    0.691E-04 &  5.3&    0.743E-01 &  5.5\\
 4&    0.390E-05 &  4.1&    0.258E-02 &  4.8\\
  \hline  
\end{tabular} \end{center}  \end{table}

In the third test,  we solve the interface problem \eqref{model-1}--\eqref{model-4} also
   on a square domain $\Omega= (-1,1)\times (-1,1)$, where $g=0$, $g_D=0$, $g_N=0$,
\a{ a=\begin{cases} 1, \quad & \tx{if} \ (x,y)\in(-\frac 13,\frac 13)^2=:\Omega_1, \\
                     \lambda , \quad & \tx{if} \ (x,y)\in \Omega\setminus \Omega_1, \end{cases} }
and
\an{\label{u-3} f=1. \tx{\quad The exact solution is unknown.} }
Thus, we plot the numerical $P_4$ solution $u_h$ on the level 4 triangular grid
  (cf. Figure \ref{f13}) in Figure \ref{fs-3}. 
      
\begin{figure}[H]
 \begin{center}\setlength\unitlength{1.0pt}
\begin{picture}(360,120)(0,0)
  \put(15,108){$G_1$:} \put(125,108){$G_2$:} \put(235,108){$G_3$:} 
  \put(0,-420){\includegraphics[width=380pt]{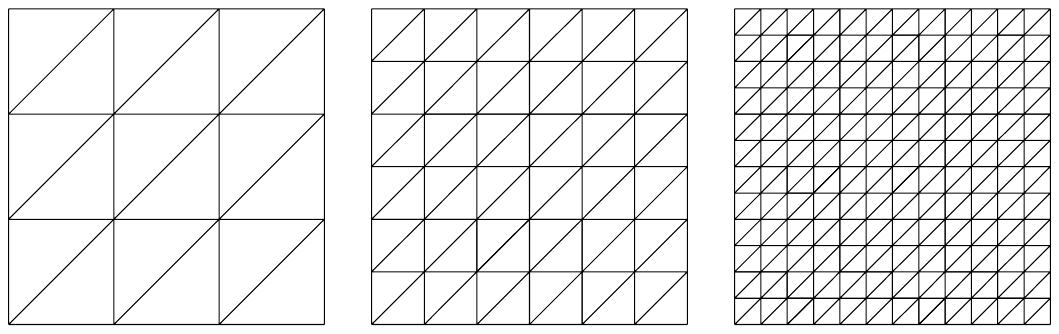}}  
 \end{picture}\end{center}
\caption{The triangular grids for the numerical solution $u_h$ for \eqref{u-3}. }\label{f13}
\end{figure}

\begin{figure}[H]
 \begin{center}\setlength\unitlength{1.0pt}
\begin{picture}(300,290)(0,0) 
  \put(0,-105){\includegraphics[width=300pt]{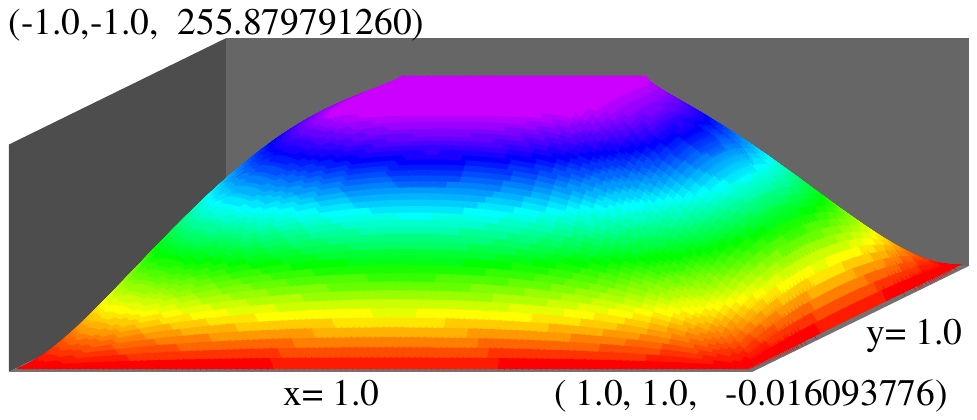}}
  \put(0,-205){\includegraphics[width=300pt]{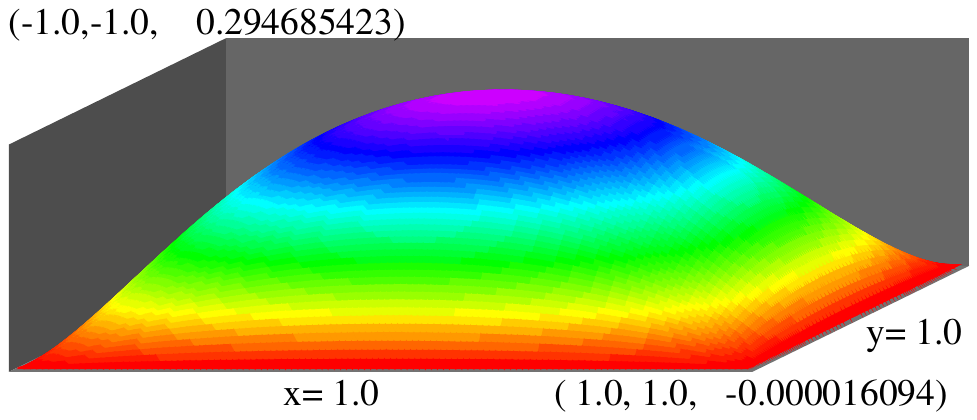}}
  \put(0,-305){\includegraphics[width=300pt]{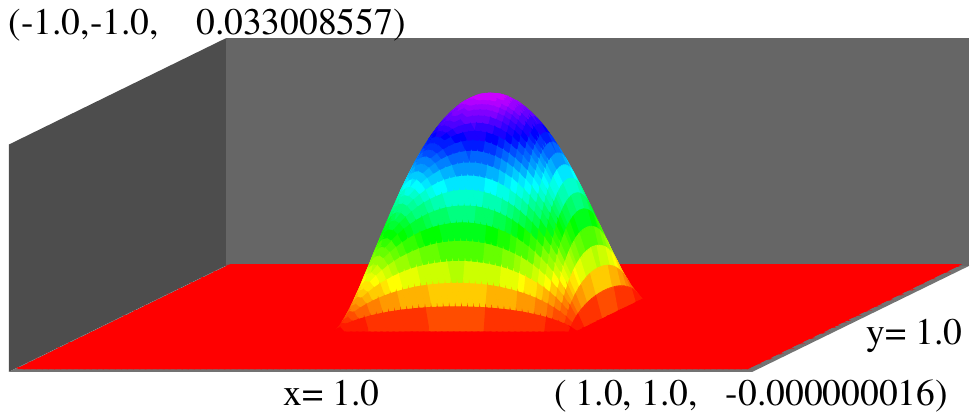}}  
 \end{picture}\end{center}
\caption{The numerical solution $u_h$ for \eqref{u-3}, when $\lambda=10^{-3}$(top), $\lambda=1$ and
  $\lambda=10^3$. }\label{fs-3}
\end{figure}

The solution in \eqref{u-3} is approximated by the weak Galerkin finite element
   $P_k$-$P_k$/$P_{k+1}$ (for $\{u_0, u_b\}$/$\nabla_w$), $k= 1,2,3$, on triangular
    grids shown in Figure \ref{f13}.
We take the difference between the $P_4$ solution  $u_4$
   (shown in Figure \ref{fs-3}) and the $P_k$ solution
  $u_h$ at two points as the errors there.
The errors and the computed orders of convergence are listed in Tables \ref{t13}--\ref{t15}. 
Roughly,  we achieved the optimal order of convergence, $O(h^{k+1})$ for $P_k$ solutions,
  in the last three tables.
  
  \vskip -.5cm
  \begin{table}[H]
  \caption{By the $P_1$-$P_1$/$P_2$ element for \eqref{u-3}  on Figure \ref{f13} grids.} \label{t13}
\begin{center}  
   \begin{tabular}{c|rr|rr}  
 \hline 
$G_i$ &  $ | (u_4 -   u_h)(0,0) | $ & $O(h^r)$ &  $ | (u_4 -   u_h)(2/3, 2/3) | $ & $O(h^r)$  \\ \hline 
&\multicolumn{4}{c}{$\lambda=10^{-3}$ }\\
 \hline
 3&    0.562E+00 &  1.8&    0.223E+00 &  3.2\\
 4&    0.180E+00 &  1.6&    0.232E-01 &  3.3\\
 5&    0.679E-01 &  1.4&    0.762E-03 &  4.9\\  \hline  
 &\multicolumn{4}{c}{$\lambda=1$ }\\
  \hline  
 3&    0.125E-02 &  2.0&    0.215E-03 &  3.2\\
 4&    0.312E-03 &  2.0&    0.236E-04 &  3.2\\
 5&    0.781E-04 &  2.0&    0.224E-05 &  3.4\\  \hline  
 &\multicolumn{4}{c}{$\lambda=10^3$ }\\
  \hline  
 3&    0.123E-02 &  1.9&    0.227E-06 &  3.1\\
 4&    0.311E-03 &  2.0&    0.254E-07 &  3.2\\
 5&    0.780E-04 &  2.0&    0.228E-08 &  3.5\\
  \hline  
\end{tabular} \end{center}  \end{table}

  \vskip -.5cm
  \begin{table}[H]
  \caption{By the $P_2$-$P_2$/$P_3$ element for \eqref{u-3} on Figure \ref{f13} grids.} \label{t14}
\begin{center}  
   \begin{tabular}{c|rr|rr}  
 \hline 
$G_i$ &  $ | (u_4 -   u_h)(0,0) | $ & $O(h^r)$ &  $ | (u_4 -   u_h)(2/3, 2/3) | $ & $O(h^r)$  \\ \hline 
&\multicolumn{4}{c}{$\lambda=10^{-3}$ }\\
 \hline 
 3&    0.801E-01 &  1.5&    0.819E-01 &  3.3\\
 4&    0.218E-01 &  1.9&    0.889E-02 &  3.2\\
 5&    0.132E-02 &  4.0&    0.145E-02 &  2.6\\ \hline  
 &\multicolumn{4}{c}{$\lambda=1$ }\\
  \hline  
 3&    0.367E-05 &  4.0&    0.864E-04 &  3.2\\
 4&    0.229E-06 &  4.0&    0.101E-04 &  3.1\\
 5&    0.143E-07 &  4.0&    0.122E-05 &  3.0\\  \hline  
 &\multicolumn{4}{c}{$\lambda=10^3$ }\\
  \hline  
 3&    0.335E-04 &  4.1&    0.872E-07 &  3.1\\
 4&    0.206E-05 &  4.0&    0.104E-07 &  3.1\\
 5&    0.128E-06 &  4.0&    0.128E-08 &  3.0\\
  \hline  
\end{tabular} \end{center}  \end{table}

  \vskip -.5cm
  \begin{table}[H]
  \caption{By the $P_3$-$P_3$/$P_4$ element for \eqref{u-3}  on Figure \ref{f13} grids.} \label{t15}
\begin{center}  
   \begin{tabular}{c|rr|rr}  
 \hline 
$G_i$ &  $ | (u_4 -   u_h)(0,0) | $ & $O(h^r)$ &  $ | (u_4 -   u_h)(2/3, 2/3) | $ & $O(h^r)$  \\ \hline 
&\multicolumn{4}{c}{$\lambda=10^{-3}$ }\\
 \hline 
 3&    0.470E-01 &  1.6&    0.589E-02 &  3.7\\
 4&    0.445E-02 &  3.4&    0.598E-03 &  3.3\\
 5&    0.367E-02 &  3.6&    0.650E-04 &  3.2\\ \hline  
 &\multicolumn{4}{c}{$\lambda=1$ }\\
  \hline  
 3&    0.831E-06 &  4.0&    0.253E-05 &  4.2\\
 4&    0.519E-07 &  4.0&    0.149E-06 &  4.1\\
 5&    0.324E-08 &  4.0&    0.912E-08 &  4.0\\  \hline  
 &\multicolumn{4}{c}{$\lambda=10^3$ }\\
  \hline   
 3&    0.760E-05 &  4.2&    0.234E-08 &  4.0\\
 4&    0.468E-06 &  4.0&    0.138E-09 &  4.1\\
 5&    0.293E-07 &  4.0&    0.592E-11 &  4.5\\
  \hline  
\end{tabular} \end{center}  \end{table}

\end{document}